\documentclass[12pt,reqno]{amsart}

\usepackage[l2tabu,orthodox]{nag}

\usepackage{hyperref}

\usepackage{a4wide}
\usepackage{amsaddr}
\usepackage{enumerate}

\newcommand{\Z}{\mathbb{Z}}

\newcommand{\R}{\mathbb{R}}
\newcommand{\C}{\mathbb{C}}
\newcommand{\F}{\mathbb{F}}

\renewcommand{\emptyset}{\varnothing}
\newcommand{\Star}{$^*$-}
\DeclareMathOperator{\Supp}{Supp}
\DeclareMathOperator{\degree}{deg}
\DeclareMathOperator{\Char}{char}

\usepackage{amsmath,amssymb,amsthm}
\theoremstyle{plain}
\newtheorem{theorem}{Theorem}[section]
\newtheorem{lemma}[theorem]{Lemma}
\newtheorem{proposition}[theorem]{Proposition}
\newtheorem{corollary}[theorem]{Corollary}
\theoremstyle{definition}
\newtheorem{definition}[theorem]{Definition}
\newtheorem{example}[theorem]{Example}
\newtheorem{conjecture}[theorem]{Conjecture}
\newtheorem{problem}{Problem}
\newtheorem{remark}[theorem]{Remark}

\title[Rings graded by torsion-free groups]{Units, zero-divisors and idempotents in rings graded by torsion-free groups}

\begin{document}

\author{Johan \"{O}inert}
\address{Department of Mathematics and Natural Sciences,\\
Blekinge Institute of Technology,
SE-37179 Karlskrona, Sweden}
\email{johan.oinert@bth.se}

\date{\today}

\subjclass[2020]{16W50, 16S34, 16S35, 20F99, 16U40, 16U60, 16U70, 16U99}
\keywords{group graded ring, torsion-free group, unique product group,
unit conjecture, zero-divisor conjecture, idempotent conjecture}

\begin{abstract}
The three famous problems concerning units, zero-divisors and idempotents in group rings of torsion-free groups, commonly attributed to I. Kaplansky, have been around for more than 60 years
and still remain open in characteristic zero.
In this article, we introduce the corresponding problems in the considerably more general context of arbitrary rings graded by torsion-free groups.
For natural reasons, we will restrict our attention to 
rings without non-trivial homogeneous zero-divisors with respect to the given grading.
We provide a partial solution to the extended problems by solving them for rings graded by unique product groups.
We also show that the extended problems exhibit the same (potential) hierarchy as the classical problems for group rings.
Furthermore, a ring which is graded by an arbitrary torsion-free group is shown to be indecomposable, and to have
no non-trivial central zero-divisor
and no non-homogeneous central unit.
We also present generalizations of the classical group ring conjectures.
\end{abstract}

\maketitle

%%%%%%%%%%%%%%%%%%%%%%%%%%%%%%
\section{Introduction}\label{Sec:Intro}
%%%%%%%%%%%%%%%%%%%%%%%%%%%%%%

%
With a few exceptions, notably \cite{H1940}, the first articles on group rings of infinite groups appeared in the early 1950s. A key person 
in that line of research was I.~Kaplansky, known for his many deep contributions to ring theory and operator algebra.
In his famous talk, given at a conference that was held on June 6-8, 1956 at Shelter Island, Rhode Island, New York, he proposed twelve problems in the theory of rings \cite{Kap56,Kap70},
one of which has become known as \emph{the zero-divisor problem} (for group rings).

Although popularized by Kaplansky, the zero-divisor problem and its corresponding conjecture
had in fact already been introduced by G. Higman in his 1940 thesis \cite[p.~77]{HigmanThesis1940} (see
also \cite[p.~112]{Sandling1981}). In \cite{HigmanThesis1940}, Higman also introduced the so-called \emph{unit problem} and the corresponding \emph{unit conjecture}.
A third problem which is closely related to the 
previous two, is \emph{the idempotent problem}.
For clarity, we now recall the exact formulation of the three problems.

\begin{problem}[Higman/Kaplansky]\label{prob:KapGR}
Let $K$ be a field, let $G$ be a torsion-free group
and denote by $K[G]$ the corresponding group ring.
\begin{enumerate}[{\rm (a)}]
	\item
	Is every unit in $K[G]$ necessarily trivial, i.e. a scalar multiple of a group element?
	\item
	Is $K[G]$ necessarily a domain?
	\item
	Is every idempotent in $K[G]$ necessarily trivial, i.e. either $0$ or $1$?
\end{enumerate}
\end{problem}

Many of the problems in Kaplansky's original list \cite{Kap56} have been solved.
The zero-divisor problem and the unit problem have been solved, in the affirmative, for several important classes of groups (see e.g. \cite{Bro1976,Cli1980,Far1976,For1973,KLM1988,Li77}). 
Significant progress has been made on the idempotent problem using algebraic as well as analytical methods (see e.g. \cite{Bur1970,For1973b} and \cite{HK2001,Mar1986,MY2002,Pus2002}).
For a thorough account of the development on the above problems during the 1970s, we refer the reader to D. S. Passman's extensive monograph \cite{Pas1977Book}.
In recent years computational approaches have been proposed as means
of attacking the zero-divisor problem (see \cite{DHJ2015,Sch2013}).
In 2021, there was a major breakthrough when G. Gardam \cite{Gardam2021} presented a counterexample to the unit conjecture in characteristic 2.
Soon thereafter, A. G. Murray \cite{Murray2021} presented counterexamples in arbitrary prime characteristic. 
Nevertheless, for a general group $G$ all three problems remain open for group rings of characteristic zero.

We should point out that the problems exhibit a (potential) hierarchy. Indeed,
an affirmative answer to the unit problem, for a fixed group $G$, implies an affirmative answer to the zero-divisor problem, which in turn implies an affirmative answer to the idempotent problem for the same group $G$ (see \cite[Rem.~1.1]{Val2002}).

In the last two decades the idempotent problem has regained interest, mainly due to its
connection with the Baum-Connes conjecture in operator algebra (see e.g \cite{Val2002}) via the so-called Kadison-Kaplansky conjecture for reduced group C\Star algebras.
The idempotent problem is also connected to the Farrell-Jones conjecture (see \cite{BLR2008}).
Moreover, W. L\"{u}ck \cite{L2002} has shown that if $G$ is a torsion-free group and $K$ is a subfield of $\C$ which
satisfies the Atiyah conjecture \cite[Conj.~10.3]{L2002} with coefficients in $K$, then
the zero-divisor problem has an affirmative answer for $K[G]$.
Altogether, this shows that Problem~\ref{prob:KapGR}, in particular in characteristic zero, remains highly relevant
to modern mathematics.\\

In this article we will consider
Problem~\ref{prob:KapGR} from a more general point of view, namely that of group graded rings.
Let $G$ be a group with identity element $e$.
Recall that a ring $R$ is said to be \emph{$G$-graded} (or \emph{graded by $G$})
if there is a collection $\{R_g\}_{g\in G}$ of additive subgroups of $R$ such that $R = \oplus_{g\in G} R_g$ and $R_g R_h \subseteq R_{gh}$, for all $g,h\in G$.
Furthermore, a $G$-graded ring $R$ 
is 
called
\emph{strongly $G$-graded} (or \emph{strongly graded by $G$}) if
$R_g R_h = R_{gh}$, for all $g,h\in G$.

Notice that the group ring $K[G]$ may be equipped with a canonical strong $G$-grading by putting
$R:=K[G]$ with $R_g := Kg$, for $g\in G$.
With this in mind, it is natural to ask whether it would make sense to extend 
Problem~\ref{prob:KapGR} to the more general context of strongly group graded rings.
It turns out that it does. In fact, we propose the following even more general set of problems
which will be the main focus of this article.

\begin{problem}\label{prob:KapSTR}
Let $G$ be a torsion-free group and let $R$ be a unital $G$-graded ring equipped with a
non-degenerate (see Definition~\ref{def:gradings})
$G$-grading such that $R_e$ is a domain.
\begin{enumerate}[{\rm (a)}]
	\item 
	Under the assumption that $\Char(R_e)=0$, is every unit in $R$ necessarily homogeneous w.r.t. the given $G$-grading?
	\item
	Is $R$ necessarily a domain?
	\item
	Is every idempotent in $R$ necessarily trivial?
\end{enumerate}
\end{problem}

This article is organized as follows.

In Section~\ref{Sec:Prel},
we record the most important notation and preliminaries concerning group graded rings
that we will need in the sequel.
In particular, 
we show that the assumptions on the grading in Problem~\ref{prob:KapSTR}
make our rings especially well-behaved 
(see Proposition~\ref{prop:componentregular} and Proposition~\ref{prop:supportsubgroup}).
In Section~\ref{Sec:UPgroups},
using a result of A. Strojnowski,
we solve Problem~\ref{prob:KapSTR}
for unique product groups (see Theorem~\ref{thm:UPGmain}).
In particular, this solves Problem~\ref{prob:KapSTR} 
in the cases where $G$ is abelian or $R$ is commutative (see Example~\ref{ex:up-groups} and Corollary~\ref{Cor:Rcomm}).
In Section~\ref{Sec:Primeness},
we show that if $G$ is an arbitrary torsion-free group
and $R$ is a $G$-graded ring 
satisfying the assumptions in Problem~\ref{prob:KapSTR},
then $R$ is a prime ring (see Theorem~\ref{prop:primeness}).
In Section~\ref{Sec:Hierarchy},
we employ the primeness result from Section~\ref{Sec:Primeness}
to show that the unit, zero-divisor and idempotent problems for group graded rings
exhibit the same (potential) hierarchy as the classical problems for group rings (see Theorem~\ref{thm:hierarchy}).
In Section~\ref{Sec:Central},
we show that for an arbitrary torsion-free group $G$ and a $G$-graded ring $R$ with a non-degenerate grading such that $R_e$ is a domain (of arbitrary characteristic),
there is no non-homogeneous central unit, no non-trivial central zero-divisor and no non-trivial central idempotent (see Theorem~\ref{thm:central}).
In Section~\ref{Sec:Quotient},
we obtain several useful results concerning gradings by quotient groups
(see e.g. Proposition~\ref{prop:quotientGN}
and Proposition~\ref{prop:GcrossedHomUnits})
which are used to solve Problem~\ref{prob:KapSTR}
for $G$-crossed products when $G$ belongs to a special class of solvable groups (see Theorem~\ref{thm:SpecialSolvable}).
In Section~\ref{Sec:Conjecture}, we formulate
a conjecture (see Conjecture~\ref{conj:GradedRingsConjecture}) which generalizes
the classical unit, zero-divisor and idempotent conjectures
for group rings.

%%%%%%%%%%%%%%%%%%%%%%%%%%%%%%
\section{Preliminaries on group graded rings}\label{Sec:Prel}
%%%%%%%%%%%%%%%%%%%%%%%%%%%%%%

Throughout this section, let $G$ be a multiplicatively written group with identity element $e$, and let $R$ be a (not necessarily unital) $G$-graded ring.

Consider an arbitrary element $r\in R$.
Notice that $r\in R$ has a unique decomposition of the form $r= \sum_{g\in G} r_g$, where $r_g\in R_g$ is zero for all but finitely many $g\in G$.
The \emph{support of $r$} is defined as the finite set
$\Supp(r) := \{g\in G \mid r_g\neq 0\}$.
If $r\in R_g$, for some $g\in G$, and $r$ is nonzero, then $r$ is said to be \emph{homogeneous of degree $g$} and we write $\degree(r)=g$.
Note that $R_e$ is a subring of $R$. If $R$ is unital, then $1_R \in R_e$ (see e.g. \cite[Prop.~1.1.1]{NVO2004}).

We shall now highlight two types of $G$-gradings which play central roles in this article.

\begin{definition}\label{def:gradings}
\begin{enumerate}[{\rm (a)}]

	\item
	$R$ is said to have a \emph{non-degenerate} $G$-grading (cf. \cite{CR1983,OL2012}) if, for each $g \in G$ and each nonzero $r_g \in R_g$, we have $r_g R_{g^{-1}} \neq \{0\}$ and $R_{g^{-1}} r_g \neq \{0\}$.
	
	\item
	$R$ is said to have a \emph{fully component regular} $G$-grading, if
	$r_g s_h \neq 0$ for any $g,h\in G$, $r_g\in R_g \setminus \{0\}$ and $s_h \in R_h \setminus \{0\}$.
	
\end{enumerate}
\end{definition}

\begin{remark}\label{rem:StrGradComReg}
(a)
Every strong $G$-grading on a  
unital 
ring $R$ is non-degenerate. Indeed, take $g\in G$ and $r_g \in R_g$,
and suppose that $r_g R_{g^{-1}} = \{0\}$.
Then $r_g = r_g 1_R \in r_g R_e = r_g R_{g^{-1}} R_g = \{0\}$.
Thus, $r_g = 0$.
Similarly, $R_{g^{-1}} r_g  = \{0\}$ implies $r_g = 0$.

(b) 
The term \emph{fully component regular} has been chosen to capture
the essence of those gradings. 
There is no immediate connection to the \emph{component regular} gradings considered by Passman in e.g. \cite[p.~16]{Pas1989Book},
which are in fact special types of non-degenerate gradings.
\end{remark}

A grading may be both non-degenerate and fully component regular, but, as the following example shows, the two notions are quite independent.

\begin{example}\label{exmp:differenceNonDegFullyCompReg}
In the following two examples, the grading group is $G:=(\Z,+)$.

(a) Consider the polynomial ring $R:=\R[t]$ in one indeterminate.
We may define a $\Z$-grading on $R$ by putting
$R_n := \R t^n$ for $n\geq 0$, and $R_n := \{0\}$ for $n<0$.
Clearly, this grading is fully component regular,
but it is not non-degenerate.

(b) Consider the ring of $2 \times 2$-matrices with real entries, $R:=M_2(\R)$.
We may define a $\Z$-grading on $R$ by putting
\begin{displaymath}
	R_0 := \left(\begin{array}{cc}
	\R & 0 \\
	0 & \R \\
\end{array}\right)
\quad\quad
R_1 := \left(\begin{array}{cc}
	0 & \R \\
	0 & 0 \\
\end{array}\right)
\quad\quad
R_{-1} := \left(\begin{array}{cc}
	0 & 0 \\
	\R & 0 \\
\end{array}\right)
\quad
\text{ and }
\quad
R_n := \left(\begin{array}{cc}
	0 & 0\\
	0 & 0 \\
\end{array}\right)
\end{displaymath}
 whenever $|n|>1$.
This grading is neither fully component regular,
nor strong, but one easily sees that it is non-degenerate.
\end{example}

Given a subgroup $H \subseteq G$ we may define the subset
$R_H := \oplus_{h\in H} R_h$ of $R$.
Notice that $R_H$ is an $H$-graded subring of $R$.
If $R$ is unital, then $R_H$ is also unital
with $1_R = 1_{R_H} \in R_e$.
The corresponding \emph{projection map} from $R$ to $R_H$ is defined by
\begin{displaymath}
	\pi_H : R \to R_H, \quad \sum_{g\in G} r_g \mapsto \sum_{h\in H} r_h
\end{displaymath}
and it is clearly additive. In fact, it is an $R_H$-bimodule homomorphism.

\begin{lemma}\label{lem:RHbimodulemap}
Let $H$ be a subgroup of $G$.
If $a\in R$ and $b\in R_H$, then we have
$\pi_H(ab)=\pi_H(a)b$ and $\pi_H(ba)=b \pi_H(a)$.
\end{lemma}

\begin{proof}
Take $a\in R$ and $b\in R_H$.
Put $a':=a-\pi_H(a)$.
Clearly, $a=a'+\pi_H(a)$ and $\Supp(a') \subseteq G \setminus H$.
If $g\in G \setminus H$ and $h\in H$, then $gh \notin H$.
Thus, $\Supp(a'b) \subseteq G \setminus H$.
Hence, $\pi_H(ab)=\pi_H((a'+\pi_H(a))b)=
\pi_H(a'b)+\pi_H(\pi_H(a)b)
=0+\pi_H(a)b=\pi_H(a)b$.
Analogously, one may show that $\pi_H(ba) = b\pi_H(a)$.
\end{proof}

\begin{lemma}\label{lem:zdunitsRHR}
Suppose that the $G$-grading on $R$ is non-degenerate.
Let $H$ be a subgroup of $G$
and let $r \in R_H$. The following assertions hold:
\begin{enumerate}[{\rm (i)}]
		\item
		$r$ is a left (right) zero-divisor in $R_H$
	if and only if $r$ is a left (right) zero-divisor in $R$;

	\item
	$r$ is left (right) invertible in $R_H$ if and only if
	$r$ is left (right) invertible in $R$.
\end{enumerate}
\end{lemma}

\begin{proof}
The ``only if'' statements are trivial. We only need to show the ``if'' statements.
The proofs of the right-handed claims are treated analogously and are therefore omitted.

(i)
Suppose that $rs=0$ for some nonzero $s\in R$.
Then, by Lemma~\ref{lem:RHbimodulemap}, $0 = \pi_H(rs)= r \pi_H(s)$.
Without loss of generality we may assume that
$\Supp(s) \cap H \neq \emptyset$ and thus $0 \neq \pi_H(s) \in R_H$.
For otherwise, we may
take some $g \in \Supp(s)$ and some nonzero $x_{g^{-1}} \in R_{g^{-1}}$
such that $e\in \Supp(s x_{g^{-1}})$.
Notice that $r(sx_{g^{-1}})=0$
and $\Supp(s x_{g^{-1}}) \cap H \neq \emptyset$.

(ii)
Suppose that $sr=1_R$ for some $s\in R$.
Then, by Lemma~\ref{lem:RHbimodulemap},
$1_{R_H}=\pi_H(1_R)=\pi_H(sr)=\pi_H(s)r$ in $R_H$.
\end{proof}

The following result highlights a crucial property of the rings appearing in Problem~\ref{prob:KapSTR}.

\begin{proposition}\label{prop:componentregular}
If the $G$-grading on $R$ is non-degenerate and $R_e$ is a domain, then the $G$-grading is fully component regular.
\end{proposition}

\begin{proof}
Take $g,h\in G$, $r_g \in R_g$ and $s_h \in R_h$.
Suppose that $r_g s_h=0$.
Then $R_{g^{-1}} r_g s_h R_{h^{-1}} = \{0\}$.
Using that $R_e$ is a domain, we get that
$R_{g^{-1}} r_g = \{0\}$ or $s_h R_{h^{-1}} = \{0\}$.
By the non-degeneracy of the grading we conclude that $r_g=0$ or $s_h=0$.
\end{proof}

\begin{remark}
There are plenty of group graded rings whose
gradings are non-degenerate but not necessarily strong.
For example crystalline graded rings \cite{NVO2008}, (nearly) epsilon-strongly graded rings \cite{NO2020,NOP2018}, and in particular Leavitt path algebras and crossed products by unital twisted partial actions.
However, it should be noted that an epsilon-strongly $G$-graded ring $R$, and in particular a partial crossed product, for which $R_e$ is a domain, is necessarily strongly graded by a subgroup of $G$.
\end{remark}

\begin{definition}
\begin{enumerate}[{\rm (a)}]

\item
The \emph{support of the $G$-grading on $R$} is defined as the set
\begin{displaymath}
	\Supp(R) :=\{g\in G \mid R_g \neq \{0\} \}.
\end{displaymath}

\item
If $\Supp(R)=G$, then the $G$-grading on $R$ is said to be \emph{fully supported}.
\end{enumerate}
\end{definition}

It is easy to see that any strong $G$-grading must be fully supported.
For a general $G$-grading, however, $\Supp(R)$ need not even be a subgroup of $G$.
As illustrated by Example~\ref{exmp:differenceNonDegFullyCompReg}(a), $\Supp(R)$ may fail to be a subgroup of $G$ even if the $G$-grading is fully component regular.
As the following result shows, we are in a rather fortunate situation.

\begin{proposition}\label{prop:supportsubgroup}
If the $G$-grading on $R$ is non-degenerate and $R_e$ is a domain,
then $\Supp(R)$ is a subgroup of $G$.
\end{proposition}

\begin{proof}
Take $g,h\in \Supp(R)$.
Using Proposition~\ref{prop:componentregular} we conclude that
$R_{gh}\supseteq R_g R_h \neq \{0\}$. Thus,
$gh\in \Supp(R)$.
Moreover, by the non-degeneracy of the grading it is clear that $R_{g^{-1}} \neq \{0\}$.
Thus, $g^{-1} \in \Supp(R)$.
This shows that $\Supp(R)$ is a subgroup of $G$.
\end{proof}

The following result follows immediately from
Proposition~\ref{prop:supportsubgroup} and
Proposition~\ref{prop:componentregular}.

\begin{corollary}\label{cor:fullyHgraded}
If the $G$-grading on $R$ is non-degenerate and $R_e$ is a domain,
then $R$ has a natural grading by the subgroup $H:=\Supp(R)$ of $G$.
This $H$-grading is non-degenerate and $R_{e_G}=R_{e_H}$ is a domain.
Moreover, the $H$-grading is fully supported and fully component regular.
\end{corollary}

%%%%%%%%%%%%%%%%%%%%%%%%%%%%%%
\section{Unique product groups}\label{Sec:UPgroups}
%%%%%%%%%%%%%%%%%%%%%%%%%%%%%%

In this section we will solve Problem~\ref{prob:KapSTR} for unique product groups (see Theorem~\ref{thm:UPGmain}).
Unique product groups were introduced by W. Rudin and H. Schneider \cite{RS1964} who called them
\emph{$\Omega$-groups}.

\begin{definition}
Let $G$ be a group.
\begin{enumerate}[{\rm (a)}]

\item
	$G$ is said to be a \emph{unique product group} if, given any two non-empty finite subsets $A$ and $B$ of $G$, there exists at least one element $g\in G$ which has a unique representation of the form $g = ab$ with $a \in A$ and $b \in B$.

\item
	$G$ is said to be a \emph{two unique products group} if, 
given any two non-empty finite subsets $A$ and $B$ of $G$ with $|A|+|B|>2$, 
there exist at least two distinct elements $g$ and $h$ of $G$ which have unique 
representations of the form $g = ab$, $h = cd$ with $a, c \in A$ and $b,d \in B$.
\end{enumerate}
\end{definition}

It is clear that every two unique products group is a unique product group. 
In 1980, Strojnowski showed that the two properties are in fact equivalent.

\begin{lemma}[Strojnowski \cite{Str1980}]\label{lemma:Strojnowski}
A group $G$ is a \emph{unique product group} if and only if it is a \emph{two unique products group}.
\end{lemma}

\begin{remark}
Every unique product group is necessarily torsion-free.
\end{remark}

We shall now state and prove the main result of this section,
and thereby simultaneously generalize
e.g. \cite[Thm.~12]{H1940},
\cite[Thm.~13]{H1940},
\cite[Thm.~3.2]{RS1964},
\cite[Prop.~3.6(a)]{Bell1987}
and \cite[Thm.~26.2]{Pas1971Book}.
Notice that, by Proposition~\ref{prop:componentregular}, the following result solves
Problem~\ref{prob:KapSTR} for unique product groups.

\begin{theorem}\label{thm:UPGmain}
Let $G$ be a unique product group and let $R$ be a unital $G$-graded ring
whose $G$-grading is fully component regular.
The following assertions hold:
\begin{enumerate}[{\rm (i)}]
	\item
	Every unit in $R$ is homogeneous;
	\item
	$R$ is a domain;
	\item
	Every idempotent in $R$ is trivial.
\end{enumerate}
\end{theorem}

\begin{proof}
	(i)
	Take $x,y\in R$ which satisfy $xy=1_R$.
	Put $A:=\Supp(x)$ and $B:=\Supp(y)$.
	By assumption $|A|$ and $|B|$ are positive.
	We want to show that $|A|=|B|=1$.
	Seeking a contradiction, suppose that
	$|A|>1$. Then $|A|+|B|>2$.
	Using that $G$ is a unique product group
	and hence, by Lemma~\ref{lemma:Strojnowski}, a two unique products group,
	there are two distinct elements $g,h \in AB$
	such that $g=ab$ and $h=cd$ with $a,c\in A$ and $b,d\in B$.
	We must have $x_a y_b =0$ or $x_c y_d =0$, since $|\Supp(xy)|=|\Supp(1_R)|=|\{e\}|=1$.
	But $R$ is fully component regular
	and hence neither of the two equalities can hold.
	This is a contradiction. We conclude that $|A|=1$, i.e. $x$ is homogeneous.
	From the equality $xy=1_R$ and the full component regularity of the grading,
	we get that $y$ is also homogeneous.
	
	(ii)
	Take two nonzero elements $x,y\in R$. Seeking a contradiction, suppose that $xy=0$.
	Using that $G$ is a unique product group, there is some $a\in \Supp(x)$
	and some $b\in \Supp(y)$ such that $x_a y_b =0$.
	By Proposition~\ref{prop:componentregular}, this is a contradiction.

	(iii)
	This follows from (ii), since $u^2=u \Leftrightarrow u(u-1_R)=0$.
\end{proof}

\begin{remark}
Notice that Theorem~\ref{thm:UPGmain}(ii) also holds if $R$ is non-unital.
We want to point out that B. Malman has already proved Theorem~\ref{thm:UPGmain}(ii)
in \cite[Lem.~3.13]{Mal2014}.
\end{remark}

\begin{remark}\label{rem:LeftRightInverses}
The proof of Theorem~\ref{thm:UPGmain}(i) yields a seemingly stronger conclusion than the one we aim to prove. But in fact, notice that for a $G$-graded ring $R$ with a fully component regular grading, the following three assertions are equivalent:
\begin{enumerate}
\item[{\rm (L)}]
Every left invertible element in $R$ is homogeneous;
\item[{\rm (R)}]
Every right invertible element in $R$ is homogeneous;
\item[{\rm (U)}]
Every unit in $R$ is homogeneous.
\end{enumerate}
\end{remark}

There is an abundance of classes of groups to which Theorem~\ref{thm:UPGmain} can be applied.

\begin{example}\label{ex:up-groups}
Typical examples of unique product groups are 
the diffuse groups (see \cite{Bow2000,KRD2016})
and in particular the right (or left) orderable groups,
including e.g.
all free groups,
all torsion-free nilpotent groups,
and hence all torsion-free abelian groups.
\end{example}

\begin{remark}\label{rem:nonUP}
For many years it was not known whether every torsion-free group necessarily had the unique product property.
However, in 1987 E. Rips and Y. Segev \cite{RS1987} presented an example of a torsion-free group without the unique product property.
Since then, a growing number of examples of torsion-free non-unique product groups have
surfaced (see e.g. \cite{AS2023,Car2014,GMS2015,Pro1988,Ste2015}).
\end{remark}

For an arbitrary torsion-free group $G$ we have that
$Z(G)$ is a torsion-free abelian group, and thus a unique product group.
Theorem~\ref{thm:UPGmain} now yields the following result.

\begin{corollary}\label{Cor:ZG}
Let $G$ be a torsion-free group and let $R$ be a unital $G$-graded ring
whose $G$-grading is non-degenerate.
If $R_e$ is a domain,
then every unit in $R_{Z(G)}$ is homogeneous,
$R_{Z(G)}$ is a domain
and every idempotent in $R_{Z(G)}$ is trivial.
\end{corollary}

\begin{corollary}\label{Cor:Rcomm}
Let $G$ be a torsion-free group and let $R$ be a unital commutative $G$-graded ring
whose $G$-grading is non-degenerate.
If $R_e$ is a domain,
then every unit in $R$ is homogeneous,
$R$ is an integral domain
and every idempotent in $R$ is trivial.
\end{corollary}

\begin{proof}
By Corollary~\ref{cor:fullyHgraded},
there is a torsion-free group $H$ such that
$R$ may be equipped with an $H$-grading
which is fully supported and fully component regular.
Take $g,h\in H$. There are nonzero homogeneous elements $r_g \in R_g$ and $r_h \in R_h$
such that $r_g r_h = r_h r_g \neq 0$. Thus, $R_{gh} \cap R_{hg} \neq \emptyset$ which yields $gh=hg$. This shows that $H$ is a torsion-free abelian group.
The result now follows from Theorem~\ref{thm:UPGmain}.
\end{proof}

Theorem~\ref{thm:UPGmain}, Corollary~\ref{Cor:ZG} and Corollary~\ref{Cor:Rcomm}
may be applied to $G$-crossed products, and in particular to group rings,
but more generally to strongly group graded rings.
We shall now apply the aforementioned theorem to a few examples of 
group graded rings whose gradings are (typically) not strong.

\begin{example}
(a)
Let $\F$ be a field
and consider \emph{the first Weyl algebra}
$R := \F\langle x,y \rangle/(yx-xy-1)$.
It is an easy exercise to show that $R$ is a domain,
but this elementary fact 
is also an immediate consequence of Theorem~\ref{thm:UPGmain}.
Indeed, notice that by assigning suitable degrees to the generators, $\degree(x):=1$ and $\degree(y):=-1$, $R$ becomes graded by the unique product group $(\Z,+)$.
Moreover, $R_0 = \F[xy]$ is a domain and the $\Z$-grading is non-degenerate.
Thus, the first Weyl algebra $R$ is a domain.

(b)
More generally, 
let $D$ be a ring,
let $\sigma := (\sigma_1,\ldots,\sigma_n)$ be a set of commuting automorphisms of $D$,
and let $a := (a_1,\ldots, a_n)$ be an $n$-tuple with nonzero entries from $Z(D)$ satisfying $\sigma_i(a_j)=a_j$ for $i\neq j$.
Given this data it is possible to define the corresponding
\emph{generalized Weyl algebra} $R:=D(\sigma,a)$ (see \cite{Bav1993} or \cite{NVO2008}).
One may show that $R$ is $\Z^n$-graded with $R_e = D$.
If $D$ is a domain, then the $\Z^n$-grading is non-degenerate
and Theorem~\ref{thm:UPGmain} yields that $R$ is a domain.
Thus, we have recovered \cite[Prop.~1.3(2)]{Bav1993}.

(c)
Any \emph{crystalline graded ring} $R := A \diamond_\sigma^\alpha G$
(see \cite{NVO2008}) is equipped with a non-degenerate $G$-grading with $R_e = A$.
If $A$ is a domain and $G$ is a unique product group, then 
$R$ is a domain by Theorem~\ref{thm:UPGmain}.
\end{example}

%%%%%%%%%%%%%%%%%%%%%%%%%%%%%%
\section{Primeness}\label{Sec:Primeness}
%%%%%%%%%%%%%%%%%%%%%%%%%%%%%%

As a preparation for Section~\ref{Sec:Hierarchy}, in this section we will give a sufficient condition for a ring $R$ graded by a torsion-free group $G$ to be prime (see Theorem~\ref{prop:primeness}).

Recall that a group $G$ is said to be an \emph{FC-group} if each $g\in G$ has only a finite number of conjugates in $G$. Equivalently, $G$ is an FC-group if $[G : C_G(g)]< \infty$ for each $g\in G$.
Given a group $G$, we define the subset
\begin{displaymath}
	\Delta(G):=\{g\in G \mid g \text{ has only finitely many conjugates in } G \}.
\end{displaymath}
It is not difficult to see that $\Delta(G)$ is a subgroup of $G$.

The following useful lemma can be shown in various ways (see e.g. \cite{Neu1951}).

\begin{lemma}[Neumann \cite{Neu1951}]\label{lem:neumann}
Every torsion-free FC-group is abelian.
\end{lemma}

The next lemma is used in the work of Passman
and can be shown by induction on the number of subgroups. 
We omit the proof and instead refer the reader to \cite[Lem.~1.2]{Pas1971Book}.

\begin{lemma}[Passman \cite{Pas1971Book}]\label{lem:FinIndSubgroup}
Let $L$ be a group and let $H_1, H_2, \ldots, H_n$ be a finite number of its subgroups. Suppose that there exists a finite collection of elements $s_{i,j} \in L$, for $i \in \{1,\ldots,n\}$ and $j \in \{1,\ldots,m\}$, such that $L = \cup_{i,j} H_i s_{i,j}$.
Then for some $k \in \{1,\ldots,n\}$ we have $[L: H_k ] < \infty$.
\end{lemma}

Now we use Passman's lemma to prove the next lemma which is crucial to this section.

\begin{lemma}\label{lem:FCprodsupport}
Let $G$ be a group and consider the subgroup $H:=\Delta(G)$.
Suppose that $F$ is a non-empty finite subset of $H$
and that $A, B$ are non-empty finite subsets of $G \setminus H$.
Let $f\in F$ and $h\in H$ be arbitrary.
There exists some
$g\in \cap_{f'\in F} C_G(f')$
such that $fh \notin g^{-1} A g B$.
\end{lemma}

\begin{proof}
Put $L := \cap_{f' \in F} C_G(f')$ and suppose that $A = \{a_1, \ldots, a_n\}$
and $B = \{b_1, \ldots, b_m\}$. Let $f\in F$ and $h\in H$ be arbitrary.
Seeking a contradiction, suppose that $fh \in g^{-1 }A g B$ for all $g\in L$.
For each $i\in \{1,\ldots,n\}$ we define the subgroup $H_i := L \cap C_G(a_i)$.

For $(i,j)\in \{1,\ldots,n\} \times \{1,\ldots,m\}$, if
$a_i$ is conjugate to $fh b_j^{-1}$ by an element of $L$,
then choose $s_{i,j} \in L$ such that
$s_{i,j}^{-1} a_i s_{i,j} = fh b_j^{-1}$.
Otherwise, choose $s_{i,j} = e$.
Notice that $L \supseteq \cup_{i,j} H_i s_{i,j}$.

Now, let $g\in L$ be arbitrary. 
By our assumption, there exist $i$ and $j$ such that
$fh = g^{-1} a_i g b_j$.
That is, $fh b_j^{-1} = g^{-1} a_i g = s_{i,j}^{-1} a_i s_{i,j}$.
From the last equality we get that $g s_{i,j}^{-1} \in L \cap C_G(a_i) = H_i$.
Thus, $g \in H_i s_{i,j}$.
Since $g$ was arbitrarily chosen,
it is clear that $L \subseteq \cup_{i,j} H_i s_{i,j}$.
This shows that $L = \cup_{i,j} H_i s_{i,j}$.

By Lemma~\ref{lem:FinIndSubgroup} there is
some $k\in \{1,\ldots,n\}$ such that $[L: H_k] < \infty$.
Consider the chains of subgroups
$G \supseteq L \supseteq H_k$ and $G \supseteq C_G(a_k) \supseteq H_k$.
Recall that $[G : C_G(r)] < \infty$ for each $r \in F$.
Thus, $[G : L] < \infty$ and hence $[G : H_k] < \infty$.
This shows that $[G : C_G(a_k)] < \infty$.
But this is a contradiction, since $a_k \in G \setminus H$.
\end{proof}

\begin{theorem}\label{prop:primeness}
Let $G$ be a torsion-free group
and let $R$ be a $G$-graded ring.
If the $G$-grading on $R$ is
non-degenerate and $R_e$ is a domain,
then $R$ is a prime ring.
\end{theorem}

\begin{proof}
By Corollary~\ref{cor:fullyHgraded},
$G':=\Supp(R)$ is a torsion-free subgroup of $G$.
Moreover, $R$ can be equipped with a non-degenerate and fully supported $G'$-grading.
Thus, we will without loss of generality assume that the $G$-grading on $R$ is fully supported.

Put $H := \Delta(G)$.
Notice that $H$ is a torsion-free FC-group.
Thus, by Lemma~\ref{lem:neumann}, $H$ is torsion-free abelian.
Theorem~\ref{thm:UPGmain} now yields that $R_H$ is a domain.

Seeking a contradiction, suppose that there are nonzero ideals $I$ and $J$ of $R$ such that $I\cdot J=\{0\}$.
By Proposition~\ref{prop:componentregular}, the $G$-grading on $R$ is fully component regular. Using this and Lemma~\ref{lem:RHbimodulemap},
it is clear that $\pi_H(I)$ and $\pi_H(J)$ are nonzero ideals of $R_H$.

Choose some
$x\in I$ and $y\in J$ such that $\pi_H(x)\neq 0$ and $\pi_H(y)\neq 0$.
Put $x':=\pi_H(x)$ and $x'':=x-x'$.
Notice that
$F:=\Supp(x') \subseteq H$ and $A:=\Supp(x'') \subseteq G \setminus H$.
Similarly, put $y':=\pi_H(y)$ and $y'':=y-y'$,
and notice that $\Supp(y') \subseteq H$ and $B:=\Supp(y'') \subseteq G \setminus H$.

Choose some $f\in F$ and $h\in \Supp(y')$.
Put $L := \cap_{f' \in F} C_G(f')$ and let $g\in L$ be arbitrary.
Choose some nonzero elements $r_g \in R_g$ and $r_{g^{-1}} \in R_{g^{-1}}$.
Then $r_{g^{-1}} x r_g \subseteq I$ and thus
\begin{align*}
0=r_{g^{-1}} x r_g \cdot y &= (r_{g^{-1}} x' r_g + r_{g^{-1}} x'' r_g) ( y' + y'') \\
&= r_{g^{-1}} x' r_g y' + r_{g^{-1}} x'' r_g y'
+ r_{g^{-1}} x' r_g y'' + r_{g^{-1}} x'' r_g y''.
\end{align*}
Now, by combining the facts that the $G$-grading on $R$ is fully supported and fully component regular, and that $R_H$ is a domain,
it is not difficult to see that
$r_{g^{-1}} x' r_g y' \neq 0$.
In fact,
$fh \in \Supp(r_{g^{-1}} x' r_g y') = \Supp(x' y') \subseteq H$.
Using that $H$ is a subgroup of $G$ which is closed under conjugation,
we notice that
$\Supp(r_{g^{-1}} x'' r_g y')\cap H = \emptyset$
and
$\Supp(r_{g^{-1}} x' r_g y'') \cap H = \emptyset$
and hence we must have
$fh \in \Supp(r_{g^{-1}} x'' r_g y'') = g^{-1} A g B$.
But $g \in L$ may be chosen arbitrarily, and thus Lemma~\ref{lem:FCprodsupport} yields a contradiction.
This shows that $R$ is prime.
\end{proof}

If $R$ is a unital ring and $x,y \in Z(R)$ are nonzero elements satisfying $x y =0$,
then $I:=x R$ and $J:=y R$ are nonzero ideals of $R$ such that $I \cdot J = \{0\}$.
Thus, we obtain the following corollary which
generalizes a conclusion which,
using results of R. G. Burns \cite{Bur1970}, is already well-known for group rings.

\begin{corollary}\label{cor:CenterDomain}
Let $G$ be a torsion-free group
and let $R$ be a unital $G$-graded ring.
If the $G$-grading on $R$ is
non-degenerate and $R_e$ is a domain,
then $Z(R)$ is an integral domain.
In particular, 
every central idempotent in $R$ is trivial.
\end{corollary}

%%%%%%%%%%%%%%%%%%%%%%%%%%%%%%
\section{A potential hierarchy between the three problems}\label{Sec:Hierarchy}
%%%%%%%%%%%%%%%%%%%%%%%%%%%%%%

For group rings (cf. Problem~\ref{prob:KapGR}) it is well-known
that an affirmative answer to the unit conjecture, for a fixed field $K$, would yield
an affirmative answer to the zero-divisor conjecture, which in turn
would yield an affirmative answer to the idempotent conjecture (see e.g. \cite[p.~12]{Val2002}). However, 
it is not known whether two (or all) of the three conjectures are equivalent.

In this section we will use the main result from Section~\ref{Sec:Primeness} to show that the corresponding problems for group graded rings (see Problem~\ref{prob:KapSTR}) exhibit the same potential hierarchy (see Theorem~\ref{thm:hierarchy}).
We begin by showing the following generalization of \cite[Lem.~13.1.2]{Pas1977Book}.

\begin{proposition}\label{prop:ReducedDomain}
Let $G$ be a torsion-free group and let $R$ be a $G$-graded ring whose $G$-grading is non-degenerate.
Then $R$ is a domain if and only if $R$ is reduced and $R_e$ is a domain.
\end{proposition}

\begin{proof}
The ``only if'' statement is trivial. We proceed by showing the ``if'' statement.
To this end, suppose that $R_e$ is a domain and that $R$ is not a domain. We need to show that $R$ is not reduced.
By Theorem~\ref{prop:primeness}, we conclude that $R$ is a prime ring.
Choose some nonzero elements $x,y \in R$ which satisfy $xy=0$.
By primeness of $R$ we have $y R x \neq \{0\}$.
Notice that $(yRx)^2 \subseteq y R xy R x = \{0\}$.
Thus, there is some nonzero $z \in yRx$ which satisfies $z^2=0$.
This shows that $R$ is not reduced.
\end{proof}

We are now ready to state and prove the main result of this section.
It generalizes e.g. \cite[Lem.~13.1.2]{Pas1977Book}.
See also \cite[Rem.~1.1]{Val2002}.

\begin{theorem}\label{thm:hierarchy}
Let $G$ be a torsion-free group
and let $R$ be a unital $G$-graded ring.
Furthermore, suppose that the $G$-grading on $R$ is
non-degenerate and that $R_e$ is a domain.
Consider the following assertions:
\begin{enumerate}[{\rm (i)}]
	\item
	Every unit in $R$ is homogeneous;
	\item
	$R$ is reduced;
	\item
	$R$ is a domain;
	\item
	Every idempotent in $R$ is trivial.
\end{enumerate}
Then {\rm (i)}$\Rightarrow${\rm (ii)}$\Rightarrow${\rm (iii)}$\Rightarrow${\rm (iv)}.
Moreover, {\rm (iii)}$\Rightarrow${\rm (ii)}.
\end{theorem}

\begin{proof}
(i)$\Rightarrow$(ii)
Suppose that every unit in $R$ is homogeneous.
Let $x\in R$ be an element which satisfies $x^2=0$.
Notice that $(1_R+x)(1_R-x)=(1_R-x)(1_R+x)=1_R$.
This shows that $1_R-x$ is a unit in $R$, and hence by
assumption $1_R-x \in R_g$ for some $g\in G$.
Put $r_g : = 1_R-x$ and $H:=\langle g \rangle$, the subgroup of $G$ generated by $g$.
Consider the subring $R_H$ whose $H$-grading is non-degenerate. 
Using that $1_R \in R_e$, we notice that $x = 1_R - r_g \in R_H$.
We claim that $R_H$ is a domain.
If we assume that the claim holds, 
then $x^2=0$ implies $x=0$ and we are done.
Now we show the claim.

Case 1 ($g=e$): By assumption, $R_H = R_e$ is a domain.

Case 2 ($g\neq e$): $H$ is an infinite cyclic group which can be ordered.
The desired conclusion follows from Theorem~\ref{thm:UPGmain}(ii).

(ii)$\Leftrightarrow$(iii) This follows from Proposition~\ref{prop:ReducedDomain}.

(iii)$\Rightarrow$(iv) This is trivial.
\end{proof}

In Section~\ref{Sec:Central}, we will record an alternative proof of (ii)$\Rightarrow$(iv) in the above theorem 
(see Corollary~\ref{cor:altproofReducedIndec}).

\begin{remark}
If $u=u^2 \in R$ is an idempotent, then $(1_R - 2u)^2 = 1_R - 4u + 4u = 1_R$.
Thus, if $2$ is invertible in $R_e$, then one can directly, without invoking a primeness argument, show that (i)$\Rightarrow$(iv) in Theorem~\ref{thm:hierarchy} by proceeding as in the proof of (i)$\Rightarrow$(ii).
\end{remark}

For a $G$-graded ring $R$ to be a domain, it is obviously also necessary for $R_e$ to be a domain.
However, as the following example shows it is possible for $R$ to have only homogeneous units without $R_e$ being a domain.

\begin{example}\label{ex:Steinberg}
If $G$ is a unique product group and $A$ is a unital commutative ring,
then the group ring $A[G]$ has only trivial units if and only if $A$ is reduced and indecomposable (see \cite[Prop.~2.1]{Steinberg2019}).

Let $G$ be a unique product group and let $A=C^\infty(\R)$ be the algebra of all smooth functions $\R \to \R$ with pointwise addition and multiplication.
Notice that $A$ is not a domain. However, $A$ is reduced and indecomposable.
Thus, $A[G]$ has only trivial units.
\end{example}

\begin{remark}\label{rem:quaternions}
(a)
While torsion-freeness of $G$ is clearly a necessary condition for a group ring $K[G]$ to be a domain,
this is not the case for strongly $G$-graded rings in general.
Indeed, consider for instance
the real quaternion algebra $\mathbb{H}$ which is a division ring,
and which is strongly graded by the finite group $\Z/2\Z \times \Z/2\Z$.

(b)
Example~\ref{ex:Steinberg} shows that
Theorem~\ref{thm:hierarchy}
fails to hold if the assumption on $R_e$ is dropped.

(c) Non-degeneracy of the grading is not a necessary condition
for a group graded ring to be a domain. To see this, consider e.g.
Example~\ref{exmp:differenceNonDegFullyCompReg}(a).
\end{remark}

%%%%%%%%%%%%%%%%%%%%%%%%%%%%%%
\section{Central elements}\label{Sec:Central}
%%%%%%%%%%%%%%%%%%%%%%%%%%%%%%

The aim of this section is to obtain a strengthening of Corollary~\ref{cor:CenterDomain} by completely solving Problem~\ref{prob:KapSTR} for central elements (see Theorem~\ref{thm:central}).

\begin{proposition}\label{prop:centralFCsupport}
Let $G$ be a group and let $R$ be a unital $G$-graded ring
whose $G$-grading is non-degenerate. Furthermore, suppose that $R_e$ is a domain and that the $G$-grading is fully supported.
If $x$ is a central element in $R$, then the subgroup of $G$ generated by $\Supp(x)$ is an FC-group.
\end{proposition}

\begin{proof}
Let $x = \sum_{g\in G} x_g$ be a central element in $R$.
Take $s\in G$. Choose some nonzero $r_s \in R_s$ and,
using Proposition~\ref{prop:componentregular}, 
notice that
\begin{displaymath}
	s \Supp(x) = \Supp(r_s x) = \Supp(x r_s) = \Supp(x) s.
\end{displaymath}
This shows that $\Supp(x)$ is closed under conjugation by elements of $G$.
Thus, by finiteness of $\Supp(x)$, we get that $\Supp(x) \subseteq \Delta(G)$.
Let $H$ be the subgroup of $G$ generated by $\Supp(x)$.
Using that $H$ is finitely generated, we conclude that $H$ is an FC-group.
\end{proof}

We now state and prove the main result of this section.

\begin{theorem}\label{thm:central}
Let $G$ be a torsion-free group
and let $R$ be a unital $G$-graded ring.
If the $G$-grading on $R$ is
non-degenerate and $R_e$ is a domain,
then the following assertions hold:
\begin{enumerate}[{\rm (i)}]
	\item
	Every central unit in $R$ is homogeneous;
	\item
	$R$ has no non-trivial central zero-divisor;
	\item
	$R$ is indecomposable, i.e. every central idempotent in $R$ is trivial.
\end{enumerate}
\end{theorem}

\begin{proof}
By Corollary~\ref{cor:fullyHgraded},
$G':=\Supp(R)$ is a torsion-free subgroup of $G$.
Moreover, $R$ can be equipped with a non-degenerate and fully supported $G'$-grading.
Thus, we will without loss of generality assume that the $G$-grading on $R$ is fully supported.

(i)
Let $x \in Z(R)$ be a unit in $R$.
Denote by $H$ the subgroup of $G$ generated by $\Supp(x)$.
Notice that, by Lemma~\ref{lem:zdunitsRHR}(ii), $x$ is a unit in $R_H$.
By Proposition~\ref{prop:centralFCsupport}, $H$ is a torsion-free FC-group.
Thus, using Lemma~\ref{lem:neumann}
and Theorem~\ref{thm:UPGmain}
we conclude that $x$ is homogeneous.

(ii)
Let $x\in Z(R)$ be nonzero.
Suppose that $xy=0$ for some $y\in R$.
Denote by $H$ the subgroup of $G$ generated by $\Supp(x)$.
By Proposition~\ref{prop:centralFCsupport}, $H$ is a torsion-free FC-group.
Using Lemma~\ref{lem:neumann}
and Theorem~\ref{thm:UPGmain}
we conclude that $R_H$ is a domain.
Thus, $x$ is not a zero-divisor in $R_H$
and by Lemma~\ref{lem:zdunitsRHR}(i) we conclude that $y=0$.

(iii)
This follows from (ii) or from Corollary~\ref{cor:CenterDomain}.
\end{proof}

\begin{remark}
(a)
Malman has essentially
proved Theorem~\ref{thm:central}(ii)
in \cite[Prop.~3.14]{Mal2014}.
Thanks to Neumann's lemma (Lemma~\ref{lem:neumann})
our proof is shorter.

(b)
Notice that we can immediately recover Corollary~\ref{Cor:Rcomm} from
Theorem~\ref{thm:central}.
\end{remark}

We record the following well-known lemma.

\begin{lemma}\label{lem:ReducedCentral}
If $R$ is a reduced ring, then every idempotent in $R$ is central in $R$.
\end{lemma}

\begin{proof}
Let $u \in R$ be an idempotent. Take any $r\in R$.
Notice that $(ur-uru)^2=0$ and $(ru-uru)^2=0$. Using that $R$ is reduced, we conclude that $ur-uru=0$ and $ru-uru=0$. Hence, $ur=ru$. This shows that $u\in Z(R)$.
\end{proof}

By combining Lemma~\ref{lem:ReducedCentral}
and Theorem~\ref{thm:central} we get the following result.

\begin{corollary}\label{cor:altproofReducedIndec}
Let $G$ be a torsion-free group
and let $R$ be a unital $G$-graded ring.
Furthermore, suppose that the $G$-grading on $R$ is
non-degenerate and that $R_e$ is a domain.
If $R$ is reduced, then every idempotent in $R$ is trivial.
\end{corollary}

Notice that the above corollary allows us to establish the implication (ii)$\Rightarrow$(iv) in Theorem~\ref{thm:hierarchy} without relying on the primeness argument from Section~\ref{Sec:Primeness}.

%%%%%%%%%%%%%%%%%%%%%%%%%%%%%%
\section{Gradings by quotient groups}\label{Sec:Quotient}
%%%%%%%%%%%%%%%%%%%%%%%%%%%%%%

In this section we will show that
Problem~\ref{prob:KapSTR} can be approached by considering 
gradings by quotient groups (see Proposition~\ref{prop:quotientGN}).
For $G$-crossed products we obtain a more explicit connection (see Proposition~\ref{prop:GcrossedHomUnits}) and as an application we generalize a result of
A. A. Bovdi for a special class of solvable groups (see Theorem~\ref{thm:SpecialSolvable}).

\begin{remark}
Let $G$ be a group and let $R$ be a $G$-graded ring.
If $N$ is a normal subgroup of $G$, then $R$ may
be viewed as a $G/N$-graded ring.
Indeed, by writing
\begin{displaymath}
	R = \oplus_{g\in G} R_g = \oplus_{C \in G/N} ( \oplus_{h \in C} R_h  )
\end{displaymath}
it is easy to see that this yields a $G/N$-grading.
\end{remark}

The proof of the following lemma is straightforward
and is therefore left to the reader.

\begin{lemma}\label{lem:quotientNonDeg}
Let $G$ be a group and let $R$ be a $G$-graded ring whose $G$-grading is non-degenerate.
If $N$ is a normal subgroup of $G$, then the canonical $G/N$-grading on $R$ is non-degenerate.
\end{lemma}

The following result generalizes \cite[Lem.~13.1.9(i)]{Pas1977Book}
and \cite[Cor.~3.6]{OW2022}.

\begin{proposition}\label{prop:quotientGN}
Let $G$ be a group and let $R$ be a unital $G$-graded ring
whose $G$-grading is non-degenerate.
If $N$ is a normal subgroup of $G$
and $R_e$ is a domain, then the following assertions hold:
\begin{enumerate}[{\rm (i)}]
	\item
	If $R_N := \oplus_{n\in N} R_n$ is a domain and $G/N$ is a unique product group, then $R$ is a domain.
	
	\item
	Suppose that $N$ is torsion-free, that $G/N$ is a unique product group,
	and that every unit in $R$ which is contained in $R_{gN}$, for some $g\in G$,
	must be homogeneous w.r.t. the $G$-grading.
	Then every unit in $R$ is homogeneous w.r.t. the $G$-grading.
\end{enumerate}
\end{proposition}

\begin{proof}
(i)
We will view $R = \oplus_{g\in G} R_g = \oplus_{C\in G/N} \left( \oplus_{h \in C} R_h \right)$ as a $G/N$-graded ring.
By Lemma~\ref{lem:quotientNonDeg} the $G/N$-grading is non-degenerate,
and by assumption $R_N$ is a domain. The desired conclusion now follows immediately from
Theorem~\ref{thm:UPGmain}.

(ii)
We begin by noticing that, by assumption, every unit in $R_N$ must be homogeneous.
Thus, by Theorem~\ref{thm:hierarchy}, $R_N$ is a domain.
Take $x,y\in R$ which satisfy $xy=1_R$.
Let $\Supp(x)$ and $\Supp(y)$ denote the support of $x$ respectively $y$,
w.r.t. the $G$-grading.
Let $\phi : G \to G/N$ denote the quotient homomorphism.
Define $a$ and $b$ to be the cardinalities of $\phi(\Supp(x))$
and $\phi(\Supp(y))$, respectively.
If $a+b>2$, then by the unique product property of $G/N$ (and Lemma~\ref{lemma:Strojnowski})
we will reach a contradiction in the same way as in the proof of Theorem~\ref{thm:UPGmain}, by instead considering the $G/N$-grading.
Thus, $a=b=1$.
This means that there is some $g \in G$ such that
$x \in R_{gN}$ and
$y \in R_{g^{-1}N}$.
Now, by assumption, both $x$ and $y$ must be homogeneous w.r.t. the $G$-grading.
\end{proof}

\begin{remark}
By taking $N=\{e\}$, notice that from Proposition~\ref{prop:quotientGN}(i)
we recover Theorem~\ref{thm:UPGmain}(ii),
and from Proposition~\ref{prop:quotientGN}(ii) we recover Theorem~\ref{thm:UPGmain}(i).
\end{remark}

Recall that a unital $G$-graded ring $R$ is said to be a \emph{$G$-crossed product}
if, for each $g\in G$, the homogeneous component $R_g$ contains an element
which is invertible in $R$ (see \cite[Chap.~1]{NVO2004}). Every $G$-crossed product is necessarily strongly $G$-graded (see e.g. \cite[Rem.~1.1.2]{NVO2004}), and in particular its $G$-grading is non-degenerate (see Remark~\ref{rem:StrGradComReg}(a)).

The following result generalizes \cite[Lem.~13.1.9(ii)]{Pas1977Book}.

\begin{proposition}\label{prop:GcrossedHomUnits}
Let $G$ be a group and let $R$ be a $G$-crossed product.
Suppose that $N$ is a torsion-free normal subgroup of $G$, that $G/N$ is a unique product group,
and that $R_e$ is a domain.
The following two assertions are equivalent:
\begin{enumerate}[{\rm (i)}]
	\item
	Every unit in $R=\oplus_{g\in G} R_g$ is homogeneous w.r.t. the $G$-grading;
	\item
	Every unit in $R_N := \oplus_{n \in N} R_n$ is homogeneous w.r.t. the $N$-grading.
\end{enumerate}
\end{proposition}

\begin{proof}
(i)$\Rightarrow$(ii)
This is trivial.

(ii)$\Rightarrow$(i)
The $G$-grading on $R$ is non-degenerate.
Thus, the first part of the proof may be carried out in the same way as the proof of Proposition~\ref{prop:quotientGN}(ii).
Indeed, for elements $x,y\in R$ which satisfy $xy=yx=1_R$ we get that $x \in R_{gN}$ and $y\in R_{g^{-1}N}$, for some $g\in G$.
Using that $R$ is a $G$-crossed product, we may choose homogeneous units $x'$ and $y'$ of degree $g^{-1}$ and $g$, respectively, such that $x'y'=y'x'=1_R$.
Notice that
\begin{displaymath}
	1_R = x' y' = x' (xy) y' = (x'x)(yy') = (yy')(x'x)
\end{displaymath}
where $x'x \in R_N$ and $yy' \in R_N$.
By assumption $x'x$ and $yy'$ are homogeneous w.r.t. the $N$-grading on $R_N$.
Using that $x'$ and $y'$ are homogeneous, we conclude that $x$ and $y$ must be homogeneous w.r.t. the $G$-grading.
\end{proof}

By invoking Theorem~\ref{thm:UPGmain}, Example~\ref{ex:up-groups} and Theorem~\ref{thm:hierarchy} we obtain the following result.

\begin{corollary}\label{cor:Nabeliannormal}
Let $G$ be a torsion-free group and let $R$ be a
$G$-crossed product for which $R_e$ is a domain.
If $N$ is an abelian normal subgroup of $G$
such that $G/N$ is a unique product group,
then every unit in $R$ is homogeneous w.r.t. the $G$-grading.
Moreover, $R$ is a domain and every idempotent in $R$ is trivial.
\end{corollary}

As an application of the above results
we will solve Problem~\ref{prob:KapSTR}
for $G$-crossed products by a special class of solvable groups
and thereby generalize \cite[Thm.~1]{Bov1960}.

\begin{theorem}\label{thm:SpecialSolvable}
Let $G$ be a group and suppose that $G$ has a finite subnormal series
\begin{displaymath}
	\langle e \rangle = G_0 \vartriangleleft G_1 \vartriangleleft \ldots \vartriangleleft G_k = G
\end{displaymath}
with quotients $G_{i+1}/G_i$ all of which are torsion-free abelian.
If $R$ is a $G$-crossed product with $R_e$ a domain,
then every unit in $R$ is homogeneous w.r.t. the $G$-grading, and $R$ is a domain.
In particular, every idempotent in $R$ is trivial.
\end{theorem}

\begin{proof}
Using that $G_0$ is a torsion-free normal subgroup of $G_1$,
that $R_{G_0} = R_e$ is a domain,
and that $G_1/G_0$ is a unique product group,
we get by Proposition~\ref{prop:GcrossedHomUnits}
that every unit in $R_{G_1}$ is homogeneous.
More generally, if $G_i$ is torsion-free and every unit in $R_{G_i}$ is homogeneous,
then, by Theorem~\ref{thm:hierarchy}, $R_{G_i}$ is a domain.
Using that $G_i$ is normal in $G_{i+1}$
and that $G_{i+1}/G_i$ is a unique product group,
Proposition~\ref{prop:GcrossedHomUnits} yields that
every unit in $R_{G_{i+1}}$ is homogeneous.
Furthermore, since both $G_{i+1}/G_i$ and $G_i$ are torsion-free, we notice that $G_{i+1}$ is also torsion-free.
By induction over $i$ we conclude that every unit in $R$ is homogeneous w.r.t. the $G$-grading.
Clearly, $G$ is torsion-free.
Thus, Theorem~\ref{thm:hierarchy}
yields
that $R$ is a domain and that every idempotent in $R$ is trivial.
\end{proof}

\begin{remark}
It is also possible to obtain
Theorem~\ref{thm:SpecialSolvable}
directly from Theorem~\ref{thm:UPGmain}.
Indeed, one can show that
the group $G$ in Theorem~\ref{thm:SpecialSolvable}
is right-ordered (see \cite[Lem.~13.1.6]{Pas1977Book}), and hence a unique product group.
\end{remark}

%%%%%%%%%%%%%%%%%%%%%%%%%%%%%%
\section{A conjecture}\label{Sec:Conjecture}
%%%%%%%%%%%%%%%%%%%%%%%%%%%%%%

Recall that, up until now, we have been able to solve Problem~\ref{prob:KapSTR}
in the affirmative in the following important cases:
\begin{itemize}
	\item When $G$ is a unique product group, including e.g. all torsion-free abelian groups (see Theorem~\ref{thm:UPGmain} and Example~\ref{ex:up-groups}).
	\item When $R$ is commutative (see Corollary~\ref{Cor:Rcomm}).
	\item For central elements (see Theorem~\ref{thm:central}).
\end{itemize}

Despite the fact that the list of torsion-free non-unique product groups is growing (see Remark~\ref{rem:nonUP}), we dare, in view of our findings,
present the following generalizations of 
Kaplansky's conjectures for group rings.

\begin{conjecture}\label{conj:GradedRingsConjecture}
Let $G$ be a torsion-free group and let $R$ be a unital $G$-graded ring whose $G$-grading is non-degenerate. If $R_e$ is a domain with $\Char(R_e)=0$, then the following assertions hold:
\begin{enumerate}[{\rm (a)}]
	\item
	Every unit in $R$ is homogeneous;
	\item
	$R$ is a domain;
	\item
	Every idempotent in $R$ is trivial.
\end{enumerate}
\end{conjecture}

\begin{remark}\label{rem:FinGen}
Using Lemma~\ref{lem:zdunitsRHR} it is not difficult to see, that in order to resolve Conjecture~\ref{conj:GradedRingsConjecture}
it is enough to consider the case where $G$ is finitely generated.
\end{remark}

\begin{remark}
In their work,
K. Dykema, T. Heister and K. Juschenko \cite{DHJ2015}
and independently
P. Schweitzer \cite{Sch2013},
identified certain classes of finitely presented torsion-free groups.
Amongst other results, they showed that in order to prove
the zero-divisor conjecture for group rings over the field of two elements,
it is sufficient to prove the conjecture for groups coming from the previously mentioned  classes of finitely presented groups.
Remark~\ref{rem:FinGen} is in line with their observations.
\end{remark}

As mentioned in Section~\ref{Sec:Intro}, the idempotent conjecture for group rings
is related to the Kadison-Kaplansky conjecture for group C*-algebras,
to the Baum-Connes conjecture and to the Farrell-Jones conjecture.
Moreover, the zero-divisor conjecture for group rings
is related to the Atiyah conjecture.
Further investigations are required to determine
which relationship (if any)
Conjecture~\ref{conj:GradedRingsConjecture}
may have to other well-known conjectures.

\end{document}